\documentclass[12pt]{article}
\usepackage{latexsym}
\usepackage{amssymb}
\usepackage{amsmath}
\usepackage{rotating}
\usepackage{multirow}
\usepackage{mathrsfs}
\usepackage{wasysym}
\usepackage{esint}
\usepackage{rawfonts}
\input{prepictex}
\input{pictex}
\input{postpictex}
\usepackage[OT2,OT1]{fontenc}
\def\cyr{%
\renewcommand\rmdefault{wncyr}%
\renewcommand\sfdefault{wncyss}%
\renewcommand\encodingdefault{OT2}%
\normalfont
\selectfont}
\DeclareMathAlphabet{\zap}{OT1}{pzc}{m}{it}
\DeclareTextFontCommand{\textcyr}{\cyr}
\def\be{\begin{equation}}
\def\ee{\end{equation}}
\def\bea{\begin{eqnarray*}}
\def\eea{\end{eqnarray*}}

\def\bcp{\mathbb C \mathbb P}
\def\CC{\mathbb C}
\def\HH{\mathbb H}

\newtheorem{main}{Theorem}

\newtheorem{lem}{Lemma}
\newtheorem{prop}{Proposition}

\newenvironment{proof}{\medskip \noindent
{\bf Proof.}}{\hfill \rule{.5em}{1em}
\\}

\def\ZZ{{\mathbb Z}}
\def\RR{{\mathbb R}}
\def\CP{{\mathbb C \mathbb P}}
\begin{document}

\title{Twistors, Hyper-K\"ahler Manifolds,\\
and Complex Moduli}

\author{Claude LeBrun\thanks{Supported 
in part by  NSF grant DMS-DMS-1510094.}\\Stony
 Brook University}

\date{January 27, 2017}
\maketitle

\hspace{1.8in}
\begin{minipage}{3.2in}
\begin{quote} {\em For my good friend and admired 
 colleague Simon Salamon, on the occasion of
his sixtieth birthday. }
\end{quote}
\end{minipage}

\bigskip 
 
 \begin{abstract} A theorem of Kuranishi \cite{kuranishi} tells us that 
 the moduli space of complex structures on any  smooth compact manifold
 is  always  {\em locally} a finite-dimensional space. {\em Globally}, however, this 
 is simply not  true;  we display  examples in which   the moduli space contains a sequence of regions 
 for which  the local dimension tends to infinity. These  examples  naturally arise 
 from  the twistor theory of hyper-K\"ahler manifolds. 
  \end{abstract}

 If $Y$ is a smooth compact manifold, the moduli space 
 $\mathfrak{M}(Y)$ of complex structures on $Y$ is defined to be the quotient of
 the set of all smooth integrable almost-complex structure $J$ on $Y$, equipped with the topology 
 it inherits from the space of  almost-complex structures,  modulo the action of the group 
 of self-diffeomorphisms of $Y$. When we focus  only on complex structures near some given 
 $J_0$, an  elaboration of Kodaira-Spencer theory \cite{KS} due to Kuranishi \cite{kuranishi}  shows
 that the moduli space is locally finite dimensional. Indeed, if $\Theta$ denotes the sheaf 
 of holomorphic vector fields on $(Y,J_0)$, Kuranishi shows that there is a family of complex structures
 parameterized by an analytic  subvariety of the unit ball in $H^{1}(Y, \Theta )$ which, up to biholomorphism,  sweeps out every 
 complex structure near $J_0$. This subvariety of $H^{1}(Y, \Theta )$ is  defined by equations taking values
 in $H^2(Y, \Theta )$, and one must then also divide by the group of complex automorphisms of 
 $(Y,J)$, which is a Lie group with Lie algebra $H^0(Y, \Theta )$. But, in any case, near a given complex structure, 
this says that  the moduli space is  a finite-dimensional object, with dimension bounded above by $h^1(Y,\Theta)$.

What we will observe here, however, is that this local finite-dimensionality can completely break down in the large:
 
 \begin{main}
 \label{crux}
 Let $X^{4k}$ be  a smooth simply connected compact manifold that admits a hyper-K\"ahler metric. 
 Then the moduli space $\mathfrak{M}$ of complex structures on $S^2 \times X$ 
 is infinite dimensional, in the following   sense: for  every 
 $N\in \ZZ^+$, there are  holomorphic {embeddings}  
 $D^N \hookrightarrow \mathfrak{M}$ of the $N$-complex-dimensional unit polydisk $D^N:=D\times \cdots \times D \subset \CC^N$ into the moduli space. 
  \end{main}

In fact, for every natural number $N$, we will construct  proper    holomorphic  submersions  $\mathfrak{Y}\to D^N$ with fibers diffeomorphic to 
 $X\times S^2$ such that no two fibers are biholomorphically equivalent. Focusing on  this  concrete  assertion  should help  avoid   confusing 
 the phenomenon   under study with 
  other   possible structural  pathologies of   the  moduli space $\mathfrak{M}$.

 Before proceeding further, it  might  help  to  clarify how our construction differs 
 from various off-the-shelf examples where Kodaira-Spencer theory produces  mirages of  moduli
 that should not  be mistaken  for the real thing. 
 Consider the Hirzebruch surfaces $F_\ell = \mathbb{P} (\mathcal{O} \oplus \mathcal{O} (\ell))\to \CP_1$. 
 These are all diffeomorphic to $S^2 \times S^2$ or $\CP_2 \# \overline{\CP}_2$, depending on whether 
 $\ell$ is even or odd. 
For $\ell > 0$, $h^1(F_\ell , \Theta_\ell)= (\ell -1) \to \infty$ and   $h^2 (F_\ell , \Theta_\ell )=0$, so it might appear that the dimension of
 the moduli space is growing without bound. However, when  these infinitesimal deformations are realized 
 by a versal family, most of the fibers always turn out to be mutually biholomorphic, because 
 $h^0( F_\ell , \Theta_\ell)= (\ell +5) \to \infty$, too, and a cancellation arises from the action of  the automorphisms of the 
 central fiber  on  the versal deformation. In fact, the $F_\ell$ represent {\em all} the complex structures
 on $S^2 \times S^2$ and $\CP_2 \# \overline{\CP}_2$; thus, while  the corresponding moduli spaces are highly non-Hausdorff, 
 they are in fact  just  $0$-dimensional. Similar phenomena also arise from  
 projectivizations of higher-rank vector bundles over $\CP_1$; even though it is easy to construct examples with  $h^1 (\Theta )\to \infty$ in
  this context, the piece  of the moduli space one  constructs  in this way is once again  non-Hausdorff and  $0$-dimensional.

 Let us now recall  that a smooth compact Riemannian manifold  $(X^{4k},g)$ is
 said to be {\em hyper-K\"ahler}  if its holonomy is a subgroup of  $\mathbf{Sp(k)}$. One then says that  a 
 hyper-K\"ahler manifold is {\em irreducible} if its holonomy is {exactly} $\mathbf{Sp(k)}$.  This in particular implies  \cite{beauville} that $X$ is simply connected. 
 Conversely, any  simply connected compact hyper-K\"ahler manifold is a Cartesian product of irreducible ones,
since its  deRham decomposition \cite{bes} cannot involve any flat factors. 
 In order to prove Theorem \ref{crux}, one therefore might as well  assume that $(X,g)$ is irreducible, 
 since any hyper-K\"ahler manifold admits complex structures, and $S^2 \times (X \times \tilde{X}) = (S^2 \times X) \times \tilde{X}$. 
Note that examples of irreducible hyper-K\"ahler $(4k)$-manifolds
are in fact known \cite{beauville,ogrady1} for every $k\geq 1$.  When $k=1$, the unique choice for $X$ is  $K3$.
For $k\geq 2$, the smooth manifold $X$ is no longer uniquely determined by $k$, but the 
 the Hilbert scheme of $k$ points on a $K3$ surface always   provides one  simple and elegant example. 
 
  The construction  we will use to prove Theorem \ref{crux} crucially involves 
  the use of twistor spaces \cite{bes,salqk}. Recall that the standard representation of $\mathbf{Sp(k)}$ on 
  $\RR^{4k} = \HH^k$ commutes with every almost-complex structure arising from  a  
quaternionic scalar  in    $S^2 \subset \Im m\, \HH$, and that every  hyper-K\"ahler manifold is therefore
K\"ahler with respect to a $2$-sphere's worth of parallel almost-complex structures. Concretely, 
if we let $J_1$, $J_2$,  and $J_3$ denote the complex structures corresponding to the 
quaternions $\mathbf{i}$, $\mathbf{j}$,  and $\mathbf{k}$, then the integrable complex structures in question 
are those given by $aJ_1 + bJ_2 + cJ_3$ for any $(a,b,c) \in \RR^3$ with $a^2+b^2+c^2=1$. We can then 
assemble these to form an integrable almost-complex structure on $X\times S^2$ by
using the round metric and 
standard orientation on $S^2$ to make it into a $\CP_1$, and then giving the 
$X$ the integrable complex structure $aJ_1 + bJ_2 + cJ_3$ determined by 
$(a,b,c)\in S^2$. For each $x\in X$, the  stereographic coordinate $\zeta = (b+ic)/(a+1)$
on $\{ x\} \times S^2$ is thus   a compatible complex coordinate system 
on the  so-called  {\em real twistor line}  $\CP_1 \subset Z$ 
near the point $(1,0,0)$ representing 
$J_1|_x$. 
We will make considerable use of 
 the fact that  the factor projection $X\times S^2\to S^2$ now  becomes a  holomorphic submersion $\varpi : Z\to \CP_1$
 with respect to the twistor complex structure, and will systematically exploit the fact that  
$\varpi$   can therefore be thought of as a family of complex structures on $X$.

 \begin{lem}
 \label{lemon} 
  Let $(X^{4k},g)$, $k\geq 1$,  be a hyper-K\"ahler manifold, 
 and let $Z$  be its twistor space. Consider the holomorphic submersion
 $\varpi : Z\to \CP_1$  as
 a family of compact complex manifolds, and   set $X_\zeta:= \varpi^{-1} (\zeta)$
 for any  $\zeta \in \CP_1 $. 
 Then the Kodaira-Spencer map 
$T^{1,0}_{\zeta_0} \CP_1  \to H^1 (X_{\zeta_0} , \mathcal{O}(T^{1,0}X_\zeta ))$
 is non-zero at every  $\zeta_0 \in \CP_1 $. 
  \end{lem}
 \begin{proof} Since we can always change our basis for the parallel complex structures on $(X,g)$ by the action of $\mathbf{SO(3)}$, 
 we may assume that  the value $\zeta_0$ of   $\zeta\in \CP_1$ at which we wish to check the claim 
 represents the  complex structure  on $X$ we have temporarily chosen to  call $J_1$. 
 Observe  that the $2$-forms $\omega_\alpha = g( J_\alpha\cdot , \cdot )$, $\alpha=1,2,3$,  are all parallel.  
 Moreover, notice   that, 
 with respect to $J_1$, the $2$-form    $\omega_1$ is just   the K\"ahler form of $g$, while   $\omega_2+i\omega_3$ is 
 a non-degenerate   holomorphic $(2,0)$-form. 
  
  By abuse of notation, we will now  use $\zeta$ to also denote a local complex coordinate on $\CP_1$,  with $\zeta=0$ representing 
the complex structure $J_1$ of interest.  Recall that the  Kodaira-Spencer map  sends $d/d\zeta$ to an element of $H^1(X, \mathcal{O}_{J_1}(T^{1,0}_{J_1}X))$ 
 that literally encodes the derivative of the complex structure $J_\zeta$ with respect to   $\zeta$. Indeed, since we already have chosen a differentiable trivialization of our
 family, this element  is represented in Dolbeault cohomology by the $(0,1)$-form $\varphi$ with values in $T^{1,0}$ 
given  by  
 $$\varphi (v)  :=  \left.\left[\frac{d}{d\zeta} J_\zeta (v^{0,1})\right]^{1,0}\right|_{\zeta=0}$$
 where  the  decomposition $T_\CC X = T^{1,0}\oplus T^{0,1}$ used here is understood to be the one determined by $J_1$. 
 Now taking $\zeta$ to specifically be the stereographic  coordinate $\zeta = \xi + i \eta$,  where $\xi = b/(1+a)$ and $\eta = c/(1+a)$, we then have 
 $$\left.\frac{d}{d\xi}J_\zeta \right|_{\zeta=0} = J_2 \quad \mbox{and} \quad \left.\frac{d}{d\eta}J_\zeta \right|_{\zeta=0} = J_3,$$
and hence 
 $$
  \left.\frac{d}{d\zeta}J_\zeta \right|_{\zeta=0}= \frac{1}{2} (J_2-iJ_3).
 $$
 Since $T^{0,1}$ is the $(-i)$-eigenspace of $J_1$,  we therefore have 
 \begin{eqnarray*}
\varphi (v) &=& \frac{1}{2}  \left[ (J_2-iJ_3) v^{0,1}  \right]^{1,0}\\
&=& \frac{1}{2}  \left[ (J_2+iJ_2J_1) v^{0,1}  \right]^{1,0}\\
&=&  \left[ J_2  (v^{0,1})  \right]^{1,0}\\
 &=& J_2 (v^{0,1}) 
 \end{eqnarray*}
 where the last step uses the fact that $J_2$ anti-commutes with $J_1$, and therefore interchanges the $(\pm i)$-eigenspaces 
 $T^{1,0}$ and 
 $T^{0,1}$ of $J_1$. 
 
On the other hand, since $\omega_2+ i \omega_3$ is a non-degenerate holomorphic  $2$-form on $(X, J_1)$, contraction with this form
induces a holomorphic  isomorphism $T^{1,0}\cong \Lambda^{1,0}$, and hence an isomorphism $H^1(X, \mathcal{O}(T^{1,0}))\cong H^1(X, \Omega^1)$.
In Dolbeault terms, the Kodaira-Spencer class $[\varphi ]$ is thus mapped by this isomorphism
 to the element of $H^{1,1}_{\bar{\partial}_{J_1}}(X)= H^1(X, \Omega^1)$
 represented by the contraction $\varphi \lrcorner (\omega_2 + i \omega_3)$. Since 
\begin{eqnarray*}
[\varphi (v^{0,1})] \lrcorner (\omega_2 + i \omega_3) &=& g([J_2+ i J_3] \varphi (v^{0,1}) , \cdot )\\
 &=& g([J_2+ i J_1J_2] J_2 (v^{0,1}) , \cdot )\\
 &=& g(-[ I + i J_1] v^{0,1} , \cdot )\\
  &=& -2i\, \omega_1( v^{0,1} , \cdot )\\
   &=& 2i\, \omega_1(\cdot ,  v^{0,1} ), 
\end{eqnarray*}
the Kodaira-Spencer class is therefore mapped  to  $2i [\omega ] \in H^{1,1}_{\bar{\partial}_{J_1}}(X)$. 
However,     since  $[\omega_1]^{2k}$ pairs with fundamental cycle $[X]$ to  yield $(2k)!$ times the total volume  of $(X,g)$, 
   $2i\, [\omega_1]$ is certainly non-zero in deRham cohomology, and  is therefore    non-zero in 
   Dolbeault cohomology, too. The Kodaira-Spencer map of such a twistor family is thus everywhere non-zero, 
   as claimed. 
    \end{proof}
     
  We next define many new complex structures on $X\times S^2$ by  generalizing a
  construction \cite{lebchern} originally  introduced  in the $k=1$ case   to solve a
  different problem. 
  Let $f : \CP_1 \to \CP_1$ be a  holomorphic map of arbitrary degree  $\ell$. 
  We  then define
  a holomorphic family $f^{*}\varpi$ 
  over ${\mathbb C \mathbb P}_{1}$ by pulling 
  $\varpi$ back  via $f$:
  $$\begin{array}{ccc}
  	f^{*} Z& \stackrel{\hat{f}}{\longrightarrow} & Z  \\
 {\scriptstyle  f^{*}\varpi}	\downarrow 
 \hphantom{{\scriptstyle f^{*}\varpi}}&  &
 {\scriptstyle  \varpi} \downarrow \hphantom{{\scriptstyle \varpi}} \\
  	{\mathbb C \mathbb P}_{1} & \stackrel{f}{\longrightarrow} & {\mathbb C \mathbb P}_{1}.
  \end{array}$$
  In other words,  if $\varGamma\subset \CP_1 \times \CP_1$ is the graph of $f$, then 
  $f^{*}Z$ is the inverse image of $\varGamma$  under $Z\times {\mathbb C \mathbb P}_{1} 
  \stackrel{\varpi\times 1}{\longrightarrow} {\mathbb C \mathbb P}_{1}
  \times {\bcp}_{1}$. Since $\varpi$ is 
  differentiably trivial, so is $ \hat{\varpi}:= f^{*}\varpi$, and 
  $ \hat{Z}:= f^{*}Z$ may therefore be viewed as
  $X\times S^2$  equipped
  with some new  complex structure $J_{f}$. 

\begin{lem} 
\label{lulu}
Let $\hat{Z}= f^*Z$ be the complex $(2k+1)$-manifold associated with a holomorphic map $f: \CP_1 \to \CP_1$ 
of degree $\ell$, and let $\hat{\varpi} = f^*\varpi$ be the associated holomorphic submersion $\hat{\varpi} = f^* \varpi$. 
Then the canonical line bundle $K_{\hat{Z}}$ is  isomorphic to $\hat{\varpi}^* \mathcal{O}(-2k\ell -2)$ as a holomorphic 
line bundle. 
\end{lem}
\begin{proof} The twistor space of any hyper-K\"ahler manifold $(X^{4k},g)$ satisfies $K_Z= \varpi^* \mathcal{O} (- 2k-2)$. 
On the other hand, the branch locus $B$ of  $\hat{f} : \hat{Z}\to Z$ is the inverse image via $\hat{\varpi}$  of 
$2\ell -2$ points in $\CP_1$, counted with multiplicity. Thus 
$$K_{\hat{Z}} = [B] \otimes \hat{f}^*K_Z \cong \hat{\varpi}^*[\mathcal{O}(2\ell -2) \otimes \mathcal{O}(\ell (-2k-2))]= \hat{\varpi}^*\mathcal{O}(-2k\ell -2),$$
as claimed.  
\end{proof}      

This now provides one cornerstone of our argument:

\begin{prop}
\label{corner}
If  $\hat{Z}= f^*Z$ is the complex $(2k+1)$-manifold arising from a simply connected  hyper-K\"ahler manifold $(X^{4k},g)$ and a holomorphic map \linebreak 
$f: \CP_1 \to \CP_1$ 
of degree $\ell$, then there is a unique holomorphic line bundle $K^{-1/(2k\ell + 2)}$ whose $(2+ 2k\ell)^{\rm th}$ tensor power is isomorphic to the 
anti-canonical line bundle. Moreover, $h^0 ( Z, \mathcal{O} (K^{-1/(2k\ell + 2)}))=2$, and the pencil of sections of this line bundle exactly
reproduces  the holomorphic map $\hat{\varpi}: \hat{Z}\to \CP_1$. Thus the holomorphic submersion $\hat{\varpi}$ is an intrinsic 
property of the compact complex manifold $\hat{Z}= (X\times S^2, J_f)$, and is uniquely  determined, up to M\"obius transformation,  by the complex structure structure $J_f$. 
\end{prop}
  \begin{proof} Because $\hat{Z}\approx X\times S^2$ is simply connected, $H^1(\hat{Z}, \ZZ_{2k\ell +2}) = 0$, and the long exact sequence 
  induced by the short exact sequence
  of sheaves 
  $$0\to \ZZ_{2k\ell +2} \to \mathcal{O}^\times \to \mathcal{O}^\times \to 0$$
  therefore guarantees that there can be at most one holomorphic line bundle $K^{-1/(2k\ell + 2)}$ whose $(2+ 2k\ell)^{\rm th}$ tensor power is
  the anti-canonical line bundle $K^{*}$. Since Lemma \ref{lulu} guarantees that $\hat{\varpi}^* \mathcal{O}(1)$ is one candidate
  for this root of $K^*$, it is therefore the unique such root. On the other hand, since  $\hat{\varpi}^* \mathcal{O}(1)$ is trivial on the compact fibers
  of $\hat{\varpi}$, any holomorphic section of this line bundle on $\hat{Z}$ is fiber-wise constant, and is therefore  the pull-back of a section of
  $\mathcal{O}(1)$ on $\CP_1$. Thus $h^0 ( Z, \mathcal{O} (K^{-1/(2k\ell + 2)}))=h^0(\CP_1, \mathcal{O}(1))=2$, and the pencil
  of sections of  $K^{-1/(2k\ell + 2)}$ thus exactly reproduces  $\hat{\varpi}: \hat{Z} \to \CP_1$. \end{proof}

    Here,  the role of the  M\"obius transformations  is of course  unavoidable.  After all, preceding
     $f$ by a M\"obius transformation will certainly result in a biholomorphic manifold! 
     
     \bigskip 
     
      Since $\hat{\varpi}$ is intrinsically determined by the complex structure of  $\hat{Z}$, its complex structure also completely determines
      those elements of $\CP_1$ at which the Kodaira-Spencer map of the family $\hat{\varpi} : \hat{Z} \to \CP_1$ vanishes; this is 
      the same as asking for  fibers for which  there is a transverse holomorphic foliation of the first formal neighborhood. Similarly,
      one can ask whether there are elements of $\CP_1$ at which the Kodaira-Spencer map vanishes to order $m$;
      this is the same as asking for fibers for which  there is a transverse holomorphic foliation of the $(m+1)^{\rm st}$ formal neighborhood.

      \begin{prop} 
      \label{stone} 
      The critical points of $f : \CP_1 \to \CP_1$, along with their multiplicities, can be reconstructed from the submersion $f^* {\varpi} : f^* {Z} \to \CP_1$. 
           \end{prop}
      \begin{proof}
      The Kodaira-Spencer map is functorial, and transforms  with respect to pull-backs like a bundle-valued
      $1$-form. Since the Kodaira-Spencer map of $\varpi$ is everywhere non-zero by Lemma \ref{lemon}, the  points
      at which the Kodaira-Spencer map of  $\hat{\varpi}=f^*\varpi$    vanishes to order $m$ are exactly those points at which the derivative
      of $f: \CP_1 \to \CP_1$ has a critical point of order $m$.
       \end{proof}
       
     Taken together, Propositions \ref{corner} and \ref{stone}  thus  imply the following: 
           
     \begin{main} 
     \label{configure}  Modulo  M\"obius transformations,  the configuration of critical points of $f : \CP_1 \to
      \CP_1$, along with their multiplicities,  
     is   an intrinsic invariant of the compact complex manifold  $\hat{Z} = f^* Z$.  
     \end{main}

By displaying suitable families 
of holomorphic  maps $\CP_1 \to \CP_1$, we will now use Theorem \ref{configure} prove Theorem \ref{crux}.
Indeed, for any  $(a_1, \ldots , a_N) \in \CC^N$ with $|a_j -2j|< 1$,
let $P_{a_1, \ldots , a_N}(\zeta)$ be the   polynomial of degree $N+6$ in the complex variable $\zeta$ defined by 
$$
P_{a_1, \ldots , a_N}(\zeta) = \int_0^\zeta t^2  (t-1)^3 (t-a_1) \cdots (t-a_N)dt,
$$
and let $f_{a_1, \ldots , a_N} : \CP_1 \to \CP_1$ be the self-map of $\CP_1 = \CC\cup \{ \infty \}$ 
obtained by extending $P_{a_1, \ldots , a_N} : \CC \to \CC$ via  $\infty \mapsto \infty$; 
in other words, 
$$
f_{a_1, \ldots , a_N} ([\zeta_1 , \zeta_2]) = [ P_{a_1, \ldots , a_N} (\zeta_1, \zeta_2) , \zeta_2^{N+6}], 
$$
where $P_{a_1, \ldots , a_N}(\zeta_1, \zeta_2)$ is the homogeneous polynomial formally defined by 
$$P_{a_1, \ldots , a_N}(\zeta_1, \zeta_2) = \zeta_2^{N+6} P_{a_1, \ldots , a_N}({\textstyle \frac{\zeta_1}{\zeta_2}}).$$
Since the constraints we have imposed on our auxiliary parameters force the complex numbers $0,1, a_1, \ldots , a_N$ to all be distinct,
the critical points of  $f_{a_1, \ldots , a_N} : \CP_1 \to \CP_1$ are just the   $a_1, \ldots , a_N$, each with multiplicity $1$, 
along with $0$, $1$, and $\infty$, which  are individually  distinguishable by  their respective multiplicities of  $2$, $3$, and $N+5$. 
Since any M\"obius transformation that fixes $0$, $1$, and $\infty$ must be the identity,  Theorem \ref{configure} implies 
that different values of the parameters $(a_1, \ldots , a_N)$, subject the constraints $|a_j-2j|< 1$, will always result in non-biholomorphic complex manifolds $\hat{Z}_{a_1, \ldots , a_N}:= f_{a_1, \ldots , a_N}^*Z$.
Thus, pulling back $\varpi: Z\to \CP_1$  via the holomorphic map 
\begin{eqnarray*}
\Phi : D^N \times \CP_1 &\longrightarrow& \CP_1\\
   (u_1, \ldots , u_n, [\zeta_1, \zeta_2] ) & \longmapsto & f_{u_1+2, \ldots , u_N + 2N} ([\zeta_1, \zeta_2])
\end{eqnarray*}
now produces  a family $\Phi^* \varpi : \Phi^* Z\to D^N$ 
of mutually non-biholomorphic  complex manifolds over the  unit polydisk $D^N\subset \CC^N$.
 Since these manifolds are all diffeomorphic to $X\times S^2$, and since   
 this works for any positive integer $N$,  Theorem \ref{crux} is therefore an immediate consequence. 

\bigskip

Of course, the above proof is set in the wide world of  compact complex manifolds, and so has little to say about   conditions prevailing 
in the tidier  realm of, say, complex algebraic varieties. In fact, one should  probably expect  the  examples described in this article to 
never be of K\"ahler type, since  there are results in this direction  \cite{lebchern} when $k=1$. It would certainly be interesting to see this definitively
established  for general $k$. 

On the other hand, the   feature  of the $k=1$ case  highlighted by  \cite{lebchern} readily  generalizes to higher dimensions;  
namely,  the Chern numbers of the complex structures $J_f$
change as we vary  the degree of $f$. Indeed,  notice the tangent bundle of $X\times S^2$ is stably isomorphic to the pull-back of the tangent 
bundle of $X$, and that $TX$ has some non-trivial Pontrjagin numbers;  for example, if we  assume for simplicity that $X$ is irreducible, we then have
  $\hat{A}(X) =k+1$. Since the fibers of $f^*\varpi$
are Poincar\'e dual to $c_1 (f^* Z) /(2k\ell +2)$, we have  $(\mathbf{c_1 \hat{A}})(f^*Z)= 2(k\ell +1)(k+1)$,
and  certain combination of the Chern numbers of $f^*(Z)$ therefore  grows linearly in  $\ell =\deg f$. 
Consequently, as $N\to \infty$, the families of complex structures
 we have constructed skip through  infinitely many connected components of the moduli space $\mathfrak{M}(X\times S^2)$. 
Is this necessary    for a complex moduli space to fail to be finite-dimensional? 

Finally, notice that the dimension of each  exhibited component of the moduli space $\mathfrak{M}(X\times S^2)$ is higher 
than what might be inferred from our construction. Indeed, we have only made use of a single
hyper-K\"ahler metric $g$ on $X$, whereas these  in practice always come in large families. 
Hyper-K\"ahler twistor spaces also carry a tautological anti-holomorphic involution, whereas their generic small deformations 
generally will not. In short, these moduli spaces are still  largely   {\em terra incognita}. 
Perhaps some interested reader will take up the challenge, and tell us  much  more about them! 

\bigskip

\bigskip

\noindent 
{\bf Acknowledgments:} This paper is dedicated to  my friend and  sometime collaborator Simon Salamon, who first  introduced
me to hyper-K\"ahler manifolds and quaternionic geometry when we were both  graduate students at Oxford. I would 
also like to thank my colleague Dennis Sullivan for   drawing  my attention to the finite-dimensionality  problem for moduli spaces. 

%   
%      
%      \bibliographystyle{siam}
%\bibliography{lebrun}

\bigskip 

\noindent
{\sc Department of Mathematics, State University of New York, Stony Brook, NY 11794-3651, USA}

\end{document}